\documentclass[a4paper, 11pt]{article}

\usepackage{mani}
\begin{document}

\title{A Lost Counterexample and a Problem on Illuminated Polytopes}
\author{Ronald F.~Wotzlaw\thanks{\DFGsupportRFW} \qquad G\"unter M.~Ziegler\thanks{\DFGsupportGMZ}} 
\date{August 12, 2009} 
\maketitle

\noindent
In a \emph{Note added in proof} to a 1984 paper,
Daniel A.~Marcus claimed to have a counterexample to his conjecture that a minimal
positively $k$-spanning vector configuration in $\R^m$ has size at most $2km$.
However, the counterexample was never published, and seems to be lost.

Independently, ten years earlier, Peter Mani in 1974 solved a problem by Hadwiger,
disproving that every “illuminated” $d$-dimensional polytope must have at least $2d$
vertices.

These two studies are related by Gale duality, an elementary linear algebra technique
devised by Micha A. Perles in the sixties. Thus, we note that Mani's study provides
a counterexample for Marcus' conjecture with exactly the parameters that Marcus had claimed. In the other direction,
with Marcus' tools we provide an answer to a problem left open by Mani: 
Could “illuminated” $d$-dimensional polytopes on a minimal number of vertices be nonsimplicial?
 
\section{Marcus' lost counterexample and Mani's problem}

A \emph{\psc $U$} in $\R^m$ is a finite 
configuration of vectors (multiples are allowed)
that positively span $\R^m$,  that is, if $U = \{ u_1, \ldots, u_n \}$ 
and $v \in \R^m$, there are real nonnegative numbers 
$\lambda_1, \ldots, \lambda_n$ such that
\[
v = \lambda_1 u_1 + \lambda_2 u_2 + \ldots + \lambda_n u_n.
\]
A \emph{\pksc{k}} is a positive spanning
vector configuration that is still positively spanning 
even if at most $k-1$ vectors are deleted from the configuration.
 
In two papers~\cite{Marcus1981}~\cite{Marcus1984}, 
dating from 1981 and 1984, Marcus studied
properties of positive $k$-spanning vector configurations. In particular, 
he was interested in upper bounds on the cardinality of \emph{minimal}
\pkscs{k}, that is, of 
\pkscs{k} that are minimal with 
respect to inclusion:

\begin{question}[Marcus~\cite{Marcus1981}~\cite{Marcus1984}; see also~\cite{McMullen1979}]
\label{question:marcus1}
What is the maximum size of a minimal \pksc{k} in $\R^m$?
\end{question}

A classical result, known as the \emph{Blumenthal--Robinson}
theorem~\cite{Blumenthal1953}~\cite{Davis1954}~\cite{Shephard1971}, states 
that for $k=1$ the exact answer is $2m$. Marcus
conjectured that the answer for the general case is $2km$. This would
clearly be best possible: The configuration that consists of  
$k$ copies of the standard basis vectors and their negatives, that is, 
$\pm e_1, \ldots, \pm e_m$, is a minimal \pksc{k} in $\R^m$.
Yet for $k\geq 2$ it is not obvious that there is a finite upper bound.
However, this was proved by Marcus; it can also
be derived from \emph{Perles' Skeleton Theorem} for (convex)
polytopes~\cite{Kalai1994}; see~\cite{Wotzlaw2009}. 

In the following, we focus on the case $k=2$, which is particularly interesting. 
It was pointed out and used by Marcus~\cite{Marcus1981}~\cite{Marcus1984} that 
via Gale diagrams the
above question translates into a question about polytopes. Indeed, 
for $k\geq 2$ every \pksc{k} on $n$ vectors is a \emph{Gale diagram} of 
a $d$-dimensional polytope $P$ on $f_0=n$ vertices, where $d = n - m - 1$.
The main relation between the vector configuration and the polytope
that we will use (which have the same size/number of vertices, but “live”
in different dimensions) is that the minimal positively-spanning 
subconfigurations of vectors correspond exactly to the complements of
the vertex sets of facets of the polytope; for details of this construction see, for example,
\cite[Sect.~5.4]{Gruenbaum1967} or \cite[Lect.~6]{Ziegler1995}.  

The fact that we look at a \emph{minimal} positively $k$-spanning vector
configuration for $k=2$ translates, via Gale duality, into the 
following special property of the polytope: For every vertex $u$ of $P$ there 
is a different vertex $v$ such that
$u$ and $v$ are not connected by an edge of $P$, that is, the 
edge between $u$ and $v$ is \emph{missing}.
In other words, no vertex is connected to all the other vertices.
 We will call a polytope with this property \emph{unneighborly}.

With this translation, Question~\ref{question:marcus1} may be rephrased as
follows:

\begin{question}[Marcus~\cite{Marcus1981}~\cite{Marcus1984}]
\label{question:marcus2}
What is the minimum number of vertices of an unneighborly $d$-polytope?
\end{question}

Marcus' conjectured bound of $4m$ for a \pksc{2} in $\R^m$ would 
imply that an unneighborly polytope in dimension $d$ has at least
$4(d+1)/3$ vertices. This is wrong, as we shall see, and it seems
that this was noticed by Marcus: In~\cite{Marcus1984},  
a ``Note added in proof'' says that he had discovered an 
unneighborly polytope of dimension $d=36$ on $f_0 = 49$ vertices.
(The bound of $4(d+1)/3$ would imply
that such a polytope would need to have at least $50$ vertices.) 
However, there is no mention of a construction for this polytope
or any kind of reference so that would help to recover this example.
McMullen, who reviewed Marcus' paper for ``Mathematical Reviews'' and
``Zentralblatt'' had not seen the counterexample either~\cite{McMullen2009p}.
In the following we will refer to this polytope as \emph{Marcus' lost counterexample}. 

As far as we know, Marcus was not aware of the work Peter Mani had done ten years earlier.
He had studied a question by Hugo Hadwiger on \emph{illuminated polytopes}. 
These are polytopes in which every vertex lies on an \emph{inner diagonal}, that is, 
a segment connecting two vertices that passes through the interior of the polytope.
The $d$-dimensional crosspolytope, that is,
the polytope generated from the points $\pm e_1, \pm e_2, \ldots, \pm e_d$ is an
example for such a polytope on $2d$ vertices.
Clearly, inner diagonals are missing edges, and thus every illuminated
polytope is also unneighborly. (McMullen~\cite{McMullen2009p} proposes to call them
\emph{strongly unneighborly}.) 
 
\begin{question}[Hadwiger~\cite{Hadwiger1972}]
\label{question:hadwiger}
What is the minimum number of vertices of an illuminated $d$-polytope?
More specifically, is it $2d$?
\end{question}

In~\cite{Mani1974}, Mani gave a remarkable complete answer to this question. Indeed, 
the conjectured bound of $2d$ turned out to be wrong, whereas the correct bound
is roughly $d + 2\sqrt{d}$. More precisely, Mani showed that
every illuminated $d$-polytope has at least
\[
M(d) \coloneqq \min \{ 2d, \unnmani \}
\]
vertices, where we set $p(d) \coloneqq \ceil{\frac{\sqrt{4d+1}-1}{2}}$
for $d \geq 1$.
McMullen~\cite{McMullen2009p} has noted that one can 
write the function $M(d)$ in the following simple form:
\[
M(d) = \min \{2d, \ceil{(\sqrt{d}+1)^2} \} = \min \{2d, d+1+\ceil{2\sqrt{d}} \}.
\]

According to Mani, every illuminated polytope
has at least $M(d)$ vertices and examples on $M(d)$ vertices exist:
They are obtained from a cyclic $d$-polytope with $d + p(d)$ vertices by 
``stacking'' a new vertex on $\ceil{d/p(d)} + 1$ well-chosen facets. 
(The operation \emph{stacking a facet} of a polytope $P$ 
is performed in the following way: Take  a new point that is beyond the 
facet we want to stack but beneath all other
facets. Then take the convex hull of $P$ and the new point. Observe that
this creates inner diagonals from the new point to each of the vertices
of $P$ that do not lie on the stacked facet.)

In particular, for large enough $d$ there are 
illuminated --- and thus also unneighborly --- polytopes on much fewer vertices
than $4(d+1)/3$. Indeed, the smallest $d$ where this occurs is $d=36$,
and Mani's illuminated polytope in dimension $d=36$ has 
$M(36)=49$ vertices, so it has exactly the same parameters as Marcus' 
lost counterexample. We will probably never know whether this is the
same polytope that Marcus had in mind: He never published on the subject again,
and when we tried to contact him via the 
California State Polytechnic University in Pomona, where he had been on
the faculty for a number of years, we received the information that
he had passed away some years ago.

We will refer to illuminated polytopes on the minimal number of $f_0=M(d)$ vertices as \emph{Mani polytopes}.
The Mani polytopes constructed by Mani himself are, by construction, \emph{simplicial}
(that is, all facets are simplices). In relation to
Question~\ref{question:hadwiger}, Mani thus asked:

\begin{question}[Mani~\cite{Mani1974}; see also Bremner \& Klee~\cite{BremnerKlee1999}]
\label{question:mani}
Are all illuminated polytopes with the minimum number of
vertices simplicial? Or are there nonsimplicial Mani polytopes?
\end{question}

Below we give a complete answer to this question: Up to dimension $5$, all
Mani polytopes are simplicial, and there is only one combinatorial type,
given by the crosspolytope. 
For every $d \geq 6$ though, we construct a nonsimplicial Mani polytope
on the minimum number of vertices. This corrects a statement
by Bremner \& Klee~\cite{BremnerKlee1999}, who had claimed that up to dimension
$7$ all extremal illuminated polytopes were crosspolytopes. 
 
Thus our 
construction that solves Mani's question is based on a Gale diagram 
construction, thus on a construction that uses Marcus' \pkscs{k}. 
In turn, Marcus' conjecture on the size of minimal \pkscs{k} is refuted
by taking Mani's viewpoint of illuminated polytopes. 

To summarize, the precise answers for Questions~\ref{question:marcus1} 
and~\ref{question:marcus2} remain
open, although Marcus' conjectured answers are refuted by Mani's construction
as well as by our construction in this paper: They yield illuminated $d$-polytopes
on roughly $d+2\sqrt{d}$ vertices, while Marcus' work implies that an unneighborly
$d$-polytope has at least  roughly $d+\sqrt{2d}$ vertices. 
Question~\ref{question:hadwiger} was
solved by Mani~\cite{Mani1974}, and Question~\ref{question:mani} is solved 
by our construction.

Marcus' conjecture for
Question~\ref{question:marcus1} is proven wrong for all $k \geq 2$  
in the first author's PhD thesis~\cite{Wotzlaw2009} 
based on Mani's construction for the case $k=2$.

We briefly introduce some notation. If $P$ is a $d$-polytope (which in the
following will always be assumed to be full-dimensional, that is, embedded
in $\R^d$), we denote by $\verts(P)$ the set of vertices and by
$f_0 = f_0(P) = |\verts(P)|$ the number of vertices. For polytope terminology
we refer to~\cite{Gruenbaum1967} and~\cite{Ziegler1995}.

\section{Unique Mani polytopes}

We show that in dimensions $1 \leq d \leq 5$ Mani polytopes are
combinatorially unique. Thus the only combinatorial type that appears is the
$d$-crosspolytope. 
Our argument is mainly based on the original results by Mani~\cite{Mani1974} and 
the subsequent simplification by Rosenfeld~\cite{Rosenfeld1974}.

\begin{defn}[Self illuminated, opposite sets]
Let $P$ be an illuminated $d$-polytope. A set of vertices $U \subseteq \verts(P)$ is said
to \emph{illuminate itself}
if for every vertex
$v \in U$ there is a vertex $u \in U$ such that 
\[
[u,v] \coloneqq \{ x \in \R^d : x = \lambda u + (1-\lambda) v, \lambda \in [0,1]
\}
\]
is an inner diagonal.

A set $W \subseteq \verts(P)$ is said to \emph{lie opposite the vertex $v \in \verts(P)$}
if for every $w \in W$ the segment $[v,w]$ is an inner diagonal and
$\verts(P) \setminus (W \cup \{ v \})$ illuminates itself.

Let $\Gamma(P) \coloneqq \max\{ |W| : \text{$W$ lies opposite some $v \in \verts(P)$} \}$.
\end{defn}

\noindent
The following is a slightly
stronger statement than the main result from~\cite{Rosenfeld1974} 
that is easily extracted from Rosenfeld's proof.
 
\begin{lemma}[Rosenfeld~\cite{Rosenfeld1974}]
\label{thm:rosenfeld}
Let $P$ be an illuminated $d$-polytope. If $\Gamma(P) = 1$, 
then $f_0 \geq 2d$ and there is a perfect matching on the inner diagonals, that
is, there are pairwise vertex-disjoint inner diagonals that cover all vertices
of $P$.
\end{lemma}

The next lemma is easily derived from results by 
Mani; just combine the statements in~\cite[Lemma $1$, Proposition $2$, and Proposition $3$]{Mani1974}.

\begin{lemma}[{Mani~\cite{Mani1974}}]
\label{lem:opposite-large}
Let $d \geq 3$, let $P$ be a Mani $d$-polytope, and assume that $\Gamma(P) \geq 2$. 
Then $f_0(P) \geq \unnmani$.
\end{lemma}

\begin{cor}
\label{cor:opposite-large}
Let $d \geq 3$, let $P$ be a Mani $d$-polytope with $\Gamma(P) \geq 2$. 
Then $d \geq 6$.
\end{cor}

\begin{proof}
By Lemma~\ref{lem:opposite-large}, the number of vertices is at least $\unnmani$.
But for $3 \leq d \leq 5$ we have that $\unnmani > 2d$. Since $P$ is a Mani polytope
we must have $d \geq 6$.
\end{proof}

\begin{theorem}
\label{thm:crosspolytopes}
For $1 \leq d \leq 5$ there is exactly one combinatorial type of
Mani $d$-polytope, namely the $d$-dimensional crosspolytope.
\end{theorem}

\begin{proof}
The cases $d=1,2$ are trivial, so
assume $d \geq 3$. Let $P$ be a Mani $d$-polytope with $3 \leq d \leq 5$. 
By Corollary~\ref{cor:opposite-large}, we have $\Gamma(P) = 1$.
Lemma~\ref{thm:rosenfeld} and the existence of the crosspolytopes
then imply that $f_0(P) = 2d$ and 
that there is a perfect matching on the inner diagonals. Since any facet
of $P$ can contain only one vertex of any inner diagonal, the set
of facets is a subset of the facets of the $d$-crosspolytope. This 
implies that $P$ \emph{is} the $d$-crosspolytope.
\end{proof}

\section{Nonsimplicial Mani polytopes}
\label{sec:mani-polytopes}

\begin{theorem}
\label{thm:mani-nonsimplicial}
There exists a nonsimplicial Mani $d$-polytope for every $d \geq 6$.
\end{theorem}

\begin{proof}
For every $d \geq 6$ we construct a nonsimplicial Mani $d$-polytope.
(Observe that for $d=6,7$ we have $\unnmani = 2d$.) 
Let $p \geq 1$, $q \coloneqq \ceil{d / p}$, and choose an $\ell$ with 
$1 \leq \ell \leq q-1$.

We construct a nonsimplicial polytope $Q$ on $d+p$ vertices 
that has $q+1$ simplex facets, such that
stacking onto these facets produces a nonsimplicial illuminated $d$-polytope.
(What we describe here is in fact a whole family of such polytopes, indexed
by the parameter $\ell$.) 

We describe $Q$ in terms of a Gale diagram $A$.
Let 
\[
B = \{  -\vectorone, e_1, \ldots, e_{p-1} \},
\] where $\vectorone$ denotes
the vector in which all entries are $1$. This is a positive basis of
$\Real^{p-1}$ of cardinality $p$.  
The vectors in $A$ are the following:
\begin{compactenum}[(1)]
\item Take $\ell$ copies of $B$, and denote them by $B_1, \ldots, B_\ell$.
\item Take $q-\ell$ copies of $-B$, and denote them by $\tilde{B}_1, \ldots, \tilde{B}_{q-\ell}$. 
\item Furthermore, take the vectors $\vectorone, -e_1, \ldots,-e_{d+p-pq-1}$.
\end{compactenum}
Then the number of vectors in $A$ is $d+p$. 

By the translation between Gale diagram and polytope combinatorics, 
every $B_i$, $i=1, \ldots \ell$,  and every $\tilde{B}_j$,
$j=1,\ldots, q-\ell$ corresponds to a
facet complement of size $p$ in $Q$, that is, to the complement of a simplex facet. If we augment
the set $\{ \vectorone, -e_1, \ldots, -e_{d+p-pq-1} \}$ to a positive basis $B'$
by taking the last $pq-d$ vectors of $\tilde{B}_1$, we get that
\[
\{ B_i : i=1, \ldots, \ell \} \cup \{ \tilde{B}_j : j = 1, \ldots, q-\ell \} \cup \{
B' \} 
\]
is a set of subconfigurations of $A$ that correspond to complements of 
simplex facets of $Q$. These complements cover all vertices of $Q$.
Thus, stacking onto the corresponding facets we obtain an illuminated polytope
$P$ on $d+p+q+1$ vertices, since every vertex will be on an inner diagonal 
to one of the stacking vertices. For $p \coloneqq p(d)$ we get that $P$ is a Mani polytope.

If $1 \leq \ell \leq q-1$, then there is a set of two opposite vectors in $A$,
which corresponds to a facet complement, so $Q$ is
nonsimplicial unless $p=2$. For $d \geq 7$, we have $p(d) \geq 3$
and we indeed get a nonsimplicial polytope. 
However, for $d=6$ we get $p = p(d) = 2$ and $q=3$. In this case, we
choose $p \coloneqq 3$ instead of $p \coloneqq p(d)$. 
Then $q = \ceil{d/p} = 2$ and we have $f_0(P) = 12 = M(6)$, that is, 
$P$ is a Mani polytope.

In both cases, the polytope $P$ is nonsimplicial, because $Q$ is nonsimplicial
and we only stack onto simplex facets. 
\end{proof}
 
We look at two examples that arise from the above description.
\begin{example}
For $d=16$, $p=p(d) = 4$, $q = 4$, and $\ell=3$, we display the result 
of the construction as an ``affine
Gale diagram'' in Figure~\ref{fig:gale-illuminated-16}, where all
points in a cluster represent distinct copies of a single point.
(See~\cite[Chapter 6]{Ziegler1995} for affine Gale diagrams 
and how they are related to ordinary Gale diagrams.) 

In this case, the polytope $Q$
has $f_0 = 20$ and five disjoint simplex facet complements of size four that
cover all vertices. This yields a nonsimplicial illuminated 
$16$-polytope on $M(16) = 25$ vertices.
\end{example}

\begin{example}
For the special case $d=6$, we get a polytope $Q$ as in the proof of
Theorem~\ref{thm:mani-nonsimplicial}
by constructing the Gale diagram in Figure~\ref{fig:gale-illuminated-6} 
with $p=3$, $q=2$, and $\ell=1$.

The Gale diagram has three disjoint positive bases that cover all vectors:
the basis $B_1 = \{ -\vectorone, e_1, e_2 \}$,  and the bases
$\tilde{B}_1 = B' = \{  \vectorone, -e_1, -e_2 \}$. These bases correspond to complements of simplex
facets of $Q$. Stacking onto these three facets produces a
nonsimplicial illuminated $6$-polytope on $M(6) = 12$ vertices.
\end{example}

\begin{figure}[t]
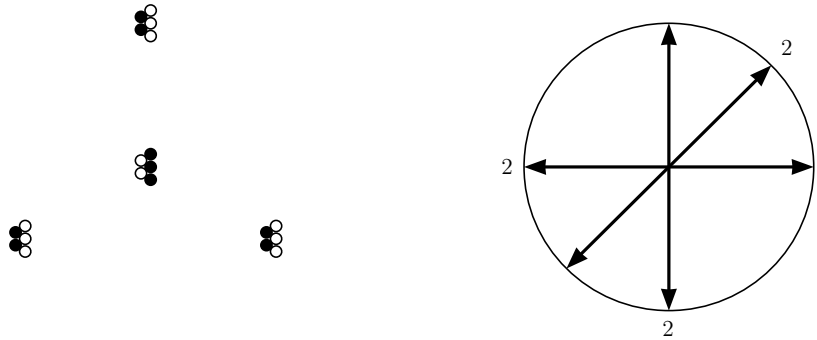

\subfigure[Stacking onto a suitable set of facets of a polytope
with this Gale diagram
yields a nonsimplicial Mani $16$-polytope.]{
\begin{minipage}[h]{.45\linewidth}
\centering
\includegraphics[scale=.6]{images/gale-illuminated-16.eps}
\label{fig:gale-illuminated-16}
\end{minipage}
}
\hfill
\subfigure[Stacking onto three well-chosen facets of a polytope 
with this Gale diagram
yields a nonsimplicial Mani $6$-polytope.]{
\begin{minipage}[h]{.45\linewidth}
 \centering
\includegraphics[scale=.6]{images/gale-illuminated-6.eps}
\label{fig:gale-illuminated-6}
\end{minipage}
}
\label{fig:gale-illuminated}
\caption{Gale diagrams of building blocks for nonsimplicial Mani
         polytopes.}
\end{figure}

\noindent
{\bf Acknowledgements.} The authors thank Raman Sanyal for bringing Mani's paper
to their attention. Peter McMullen provided valuable comments
during a number of email exchanges in 2008 and 2009.

%

 {\small
\vspace{3ex}

\noindent
 {\scshape Ronald F.~Wotzlaw, MA 6-2, Inst.\ Mathematics, TU
 Berlin, 10623 Berlin, Germany}

 \emph{E-mail address:} {\tt wotzlaw@math.tu-berlin.de}

 \vspace{1ex}

\noindent
 {\scshape G\"unter M.~Ziegler, MA 6-2, Inst.\ Mathematics, TU
 Berlin, 10623 Berlin, Germany}

 \emph{E-mail address:} {\tt ziegler@math.tu-berlin.de}

 }

\end{document}